\documentclass[a4paper,oneside,11pt]{article}
\author{Federico Venturelli}
\date{\today}
\title{Gysin morphisms for non-transversal hyperplane sections with an application to line arrangements}

\usepackage[a4paper,top=3cm,bottom=3cm,left=2.4cm,right=2.4cm,bindingoffset=7mm]{geometry}
\usepackage[T1]{fontenc}
\usepackage[english]{babel}
\usepackage{amsmath,amssymb,amsthm,amsfonts}
\usepackage{mathtools}
\usepackage[all,cmtip]{xy} 
\usepackage{indentfirst}
\usepackage{graphicx}
\usepackage{enumerate}
\usepackage{verbatim}
\usepackage{mathrsfs}
\usepackage{xcolor}
\usepackage{multicol,float}
\usepackage{caption}
\usepackage{tikz}
\usetikzlibrary{matrix}
\usetikzlibrary{arrows}
\usetikzlibrary{graphs}
\usetikzlibrary{backgrounds}
\usetikzlibrary{fadings}
\usepackage{url}
\usepackage{bbm}
\usepackage[bookmarksdepth=5]{hyperref}
\usepackage{changepage}

\newcommand{\sq}{\square}
\newcommand{\tr}{\triangle}
\newcommand{\bu}{\bullet}
\newcommand{\al}{\alpha}
\newcommand{\be}{\beta}
\newcommand{\eps}{\varepsilon}
\newcommand{\la}{\lambda}
\newcommand{\KK}{\mathbb{K}}
\newcommand{\PP}{\mathbb{P}}
\newcommand{\NN}{\mathbb{N}}
\newcommand{\ZZ}{\mathbb{Z}}
\newcommand{\QQ}{\mathbb{Q}}

\newcommand{\CC}{\mathbb{C}}
\newcommand{\Ab}{\overline{\A}}
\newcommand{\A}{\mathcal{A}}
\newcommand{\C}{\mathcal{C}}
\newcommand{\F}{\mathcal{F}}
\newcommand{\G}{\mathcal{G}}
\newcommand{\N}{\mathcal{N}}
\newcommand{\X}{\mathcal{X}}
\newcommand{\floor}[1]{\lfloor #1 \rfloor}

\theoremstyle{plain}
\newtheorem{thm}{Theorem}[section]
\newtheorem*{thm*}{Theorem}
\newtheorem{thmI}{Theorem}

\newtheorem*{prob*}{Problem}
\newtheorem{conj}{Conjecture}
\newtheorem*{conj*}{Conjecture}
\newtheorem{lem}[thm]{Lemma}
\newtheorem{prop}[thm]{Proposition}
\theoremstyle{definition}
\newtheorem{defn}[thm]{Definition}
\newtheorem{rmk}[thm]{Remark}

\numberwithin{equation}{section}

\begin{document}

\maketitle

\begin{abstract}
We prove the existence of Gysin morphisms for hyperplane sections that may not satisfy the usual hypotheses of the Lefschetz hyperplane theorem. As an application, we show the triviality of the Alexander polynomial of a particular class of non-symmetric line arrangements, thus providing positive evidence for a conjecture of Papadima and Suciu.
\end{abstract}

\section*{Introduction}

We work over the field of complex numbers $\CC$. Let $X$ be a closed subvariety of $\PP^n$ and $Y:=X\cap H$ be a hyperplane section; in order to compare the cohomology groups of $X$ and $Y$ one needs some form of the Lefschetz hyperplane theorem. The typical situations in which this theorem can be used are:

\begin{enumerate}[(a)]
\item (Classical version) Both $X$ and $Y$ are smooth.
\item (Modern formulation) There exists a Whitney stratification $\A$ of $\PP^n$ such that $X=\cup_{i=1}^k A_i$ with $A_i\in\A$ and $H$ is transversal to $A_i$ for $i=1,\ldots,k$ (see for example \cite[Theorem 1.6.5]{D1} and references therein). 
\end{enumerate}
In both cases above, the comparison is provided by Gysin morphisms $H^{k}(Y)\rightarrow H^{k+2}(X)$ which are isomorphisms for $k>\dim(Y)$ and a surjection for $k=\dim(Y)$.

In this paper we prove a version of the Lefschetz hyperplane theorem that allows one to find such Gysin morphisms even in some cases in which neither (a) nor (b) are satisfied. Key to our result are cubical hyperresolutions of varieties, which can be thought of as a way of resolving a variety `at all levels'. Indeed, when computing a resolution one usually stops when the exceptional divisor has simple normal crossings, but one could go on and resolve the singularities of the exceptional divisor and so on. Cubical hyperresolutions are a precise formalisation of this `inductive resolution' procedure. 

Cubical hyperresolutions were used in \cite{GNPP} to prove another version of the Lefschetz hyperplane theorem, which was the main motivation for our work. The precise statement is the following:

\begin{thmI}\label{thm:Guillen}
\cite[Corollaire III.3.12]{GNPP} Let $X$ be a quasi-projective complex variety and $Y$ be a hyperplane section of $X$ satisfying the following hypotheses:

\begin{enumerate}[(i)]
\item There exists an augmented $m$-cubical hyperresolution $X_{\sq}\rightarrow X$ such that $Y_{\sq}:=X_{\sq}\times_{X}Y$ is a cubical hyperresolution of $Y$.
\item For any $\al\in\sq_m^*$, there exists a closed immersion $Y_{\al}\hookrightarrow X_{\al}$ of codimension $1$.
\end{enumerate}
Then there exist Gysin morphisms between de Rham cohomology groups

\begin{equation*}
H_{DR}^k(Y)\rightarrow H_{DR}^{k+2}(X)
\end{equation*}
which are isomorphisms for $k>\dim(Y)$ and a surjection for $k=\dim(Y)$.
\end{thmI}

One can check that conditions (i) and (ii) above are satisfied if either (a) or (b) is satisfied.

The main result of this paper is the following:

\begin{thmI}\label{thm:risultato}
Let $X\subset\PP^n$ be a quasi-projective variety with singular locus $\Sigma_X$, $H\subset\PP^n$ be a hyperplane and $Y:=X\cap H$ be the corresponding hyperplane section with singular locus $\Sigma_Y$. Denote by $\tilde{X}$ and $\tilde{Y}$ resolutions of $X$ and $Y$, and by $D_X$ and $D_Y$ the corresponding exceptional divisors. 

Assume that:

\begin{enumerate}[(i)]
\item There exist $m$-cubical hyperresolutions $H(Y)_{\sq}$ and $H(X)_{\sq}$ of the resolution squares $S(Y)$ and $S(X)$ associated to $Y$ and $X$ and a closed immersion $H(Y)_{\sq}\hookrightarrow H(X)_{\sq}$.
\item There exists an $m_1$-cubical hyperresolution ${D_Y}_{\sq}$ (resp. ${D_X}_{\sq}$) of $D_Y$ (resp. $D_X$) and an $m_2$-cubical hyperresolution ${\Sigma_Y}_{\sq}$ (resp. ${\Sigma_X}_{\sq}$) of $\Sigma_Y$ (resp. $\Sigma_X$) such that:
\begin{enumerate}[(I)]
\item For all $I\in\sq^*_{m_1}$ there exists a closed immersion ${D_Y}_I\hookrightarrow {D_X}_I$ which restricts to a codimension one closed immersion on each irreducible component of ${D_Y}_I$.
\item There exists $c\in\{0,1\}$ such that for all $I\in\sq^*_{m_2}$ there exists a closed immersion ${\Sigma_Y}_I\hookrightarrow {\Sigma_X}_I$ which restricts to a codimension $c$ closed immersion on each irreducible component of ${\Sigma_Y}_I$.
\end{enumerate}
\end{enumerate}
We have the following:

\begin{enumerate}
\item If $c=1$ then there exist Gysin morphisms $H^k(Y)\rightarrow H^{k+2}(X)$ which are isomorphisms for $k>\dim(Y)$ and a surjection for $k=\dim(Y)$.
\item if $c=0$ then the conclusion of point 1. holds for $k>2\dim(\Sigma_X)+1$.
\end{enumerate}
\end{thmI}

For the terminology regarding cubical hyperresolutions and cubical varieties, as well as for the definition of (algebraic) de Rham cohomology groups, we refer to Section 1.

Observe that if $c=1$ then we obtain the same Gysin morphisms provided by the usual Lefschetz hyperplane theorem; one can check that if (b) holds then we are in this situation.

We were able to find an example in which the hypotheses of Theorem \ref{thm:risultato} are satisfied, but (a), (b) and those of Theorem \ref{thm:Guillen} are not. Our example relies on the fact that both $X$ and $Y$ have isolated and very simple singularities, which makes it easy to control the cubical hyperresolutions; this suggests that as the dimension and complexity of $\Sigma_X$ and $\Sigma_Y$ increase finding examples in which conditions (i) and (ii) of Theorem \ref{thm:risultato} are satisfied becomes much harder.

Using Theorem \ref{thm:risultato} we were able to prove the triviality of the Alexander polynomial of a particular class of line arrangements in $\PP^2$. The study of the Alexander polynomial of line arrangements is a vast subject, lying at the crossroads of topology, combinatorics and geometry; for this reason, in this introduction and in Section 3 we focus only on the aspects we care the most about.

The Alexander polynomial $\Delta_C$ of a reduced plane curve $C=V(f)$ of degree $d$ is the characteristic polynomial of the algebraic monodromy acting on $H^1(F,\CC)$, where $F$ is any fibre of the Milnor fibration $f:\CC^3\setminus f^{-1}(0)\rightarrow\CC^*$; it is known that $\Delta_C$ depends on the type and relative position of the singular points of $\CC$, and that it can be written as

\begin{equation*}
\Delta_C(t)=(t-1)^{r-1}\prod_{1<k|d}\Phi_k(t)^{e_k}=(t-1)^{r-1}q(t)
\end{equation*} 
where $r$ is the number of irreducible components of $C$ and $\Phi_k$ is the $k$-th cyclotomic polynomial. When $f$ factors into linear forms $C$ is called a line arrangement, and is usually denoted by $\Ab$. Such curves have been intensively studied by mathematicians interested in the Alexander polynomial, as one hopes to relate the latter to the combinatorial structure of the line arrangement encoded in its intersection semilattice $L(\Ab)$; in particular, it is natural to ask the following questions:

\begin{enumerate}
\item Does $\Delta_{\Ab}$ depend only on $L(\Ab)$? 
\item Are there necessary or sufficient conditions in order to have $q(t)\neq 1$ that depend only on $L(\Ab)$?
\end{enumerate}

Regarding question 1., in \cite[Conjecture 1.9]{PS} Papadima and Suciu formulated the following conjecture:

\begin{conj*}
The Alexander polynomial of a line arrangement $\Ab$ has the form

\begin{equation*}
\Delta_{\Ab}(t)=(t-1)^{|\Ab|-1}\Phi_3(t)^{e_3}[\Phi_2(t)\Phi_4(t)]^{e_4}
\end{equation*}
where $e_3=\beta_3(\Ab)$ and $e_2=e_4=\beta_2(\Ab)$. 
\end{conj*}

The $\beta_i(\Ab)$ are the modular Aomoto-Betti numbers of $\Ab$, and they depend only on $L(\Ab)$ and $i$ (see \cite[Section 3]{PS}).

As for question 2. many examples in the literature suggest that in order for $q(t)$ to be different from $1$ it is necessary that $\Ab$ admits a multinet; the latter is a purely combinatorial notion that may be thought of as a formalisation of the idea of `highly symmetric arrangement' (see \ref{defn:multinets}). 

We prove the following:

\begin{thmI}\label{thm:AP}
Let $\Ab$ be a line arrangement s.t. any line passes through at least one of two given points $P_1$ and $P_2$. Then the Alexander polynomial of $\Ab$ is trivial, i.e. $\Delta_{\Ab}(t)=(t-1)^{|\Ab|-1}$.
\end{thmI}

Since arrangements of this type do not admit multinets, this reinforces the idea that the existence of multinets is necessary in order to have $q(t)\neq 1$; moreover, our result is consistent with the conjecture by Papadima and Suciu, as $\beta_2(\Ab)=\beta_3(\Ab)=0$. 

The key steps of the proof of Theorem \ref{thm:AP} are the following:

\begin{enumerate}
\item We associate to $\Ab$ a threefold $X\subset\PP^4$ and a fibred threefold $\psi:X'\rightarrow \PP^1$ such that any fibre of $\psi$ is isomorphic to a hyperplane section $Y$ of $X$. Then we choose a generic fibre $Y$ of $\psi$, we explicitly compute the monodromy action $\phi$ on it and we show that $X$ and $Y$ satisfy the hypotheses of Theorem \ref{thm:risultato}.
\item We prove that the surjective Gysin morphism $H^2(Y)\twoheadrightarrow H^4(X)$ remains surjective if we restrict the domain to the fixed part of $H^2(Y)$ under the action of the algebraic monodromy $T_{\phi}$ (Proposition \ref{prop:monodromy}) and that it gives a surjective morphism between the primitive parts (Lemma \ref{lem:primitive}).
\item We bound the dimension of $H^2(Y)^{T_{\phi}}$ using the inclusion relation between the Hodge and polar filtration on $H^{\bu}(\PP^3\setminus Y)$ and the explicit description of the graded pieces of the latter in terms of differential forms; we then use a Thom-type result (Lemma \ref{lem:ThomIso}) to deduce from this bound our result on $\Delta_{\Ab}$.
\end{enumerate}

The paper is organised as follows. In Section 1 we give a brief account of the theory of cubical hyperresolutions, following \cite{GNPP} and \cite[Section 5]{PeSt}, showing in particular a sketch of their standard construction. Section 2 is devoted to the proof of Theorem \ref{thm:risultato}. Section 3 deals with the Alexander polynomial of curves and line arrangements, presenting the definition of multinet as well as some interesting known results, and it is meant to give the reader an idea of how the combinatorics of $\Ab$ can affect $\Delta_{\Ab}$ as well as of how Theorem \ref{thm:AP} fits in the picture; it closes with some sparse facts that we shall need in Section 4, which are placed here for lack of a better alternative. Lastly, Section 4 constitutes the proof of Theorem \ref{thm:AP}. 

\section{Cubical hyperresolutions and de Rham cohomology}

\begin{defn}
\begin{enumerate}
\item The $n$\textit{-semisimplicial category} is the category $\tr_n$ with objects the sets $[m]:=\{0,\ldots,m\}$ for $0\leq m\leq n$ and with morphisms the strictly increasing maps $[m]\rightarrow [m']$. The $n$\textit{-cubical category} is the category $\sq_n$ with objects the subsets of $[n-1]$ and with $\text{Hom}(I,J)$ consisting of a single element if $I\subset J$ and empty otherwise. We denote by $\sq_n^*$ the full subcategory of $\sq_n$ whose objects are the non-empty subsets of $[n-1]$
\item If $\C$ is any category, an $n$\textit{-cubical $\C$-object} is a contravariant functor $K_{\sq}:\sq_n\rightarrow\C$, and morphisms between such objects are morphisms of the corresponding functors; $K^*_{\sq}$ denotes the restriction of $K_{\sq}$ to $\sq_n^*$. Similarly, we can define $n$\textit{-semisimplicial $\C$-objects} $K_{\bu}$ and morphisms thereof. We will use the notations $K_I:=K_{\sq}(I)$ and $K_m:=K_{\bu}([m])$.
\item If $S$ is any object in $\C$, the \textit{constant} $n$\textit{-cubical $\C$-object} $S$ is the contravariant functor $S_{\sq}:\sq_n\rightarrow\C$ such that $S_I=S$ for all $I\in\sq_n$, with all morphisms $S_I\rightarrow S_J$ given by the identity of $S$. An \textit{augmentation} of an $n$-cubical $\C$-object $K_{\sq}$ to $S$ is a morphism of $n$-cubical $\C$-objects $K_{\sq}\rightarrow S$. If we replace $\sq_n$ by $\tr_n$, we obtain \textit{constant $n$-semisimplicial $\C$-objects} and augmentations thereof.
\end{enumerate}
\end{defn}

The next observations will be useful in what follows:

\begin{rmk}\label{rmk:cubobj}
\begin{enumerate}
\item If $X_{\sq}$ is an $n$-cubical $\C$-object we can associate to it the augmented $n$-cubical $\C$-object $\eps:X_{\sq}\rightarrow X_{\emptyset}$; sometimes we will call this augmentation the \textit{natural augmentation}.
\item Any $(n+1)$-cubical $\C$-object $X_{\sq}$ can be considered as a morphism $Y_{\sq}\rightarrow Z_{\sq}$ of $n$-cubical $\C$-objects by setting $Z_I:=X_I$ and $Y_I:=X_{I\cup\{n\}}$ for $I\in\sq_n$; in particular, a $1$-cubical $\C$-object is the datum of two objects $X,Y\in\C$ and a morphism $f:X\rightarrow Y$ between them.
\item To any $(n+1)$-cubical $\C$-object $X_{\sq}$ we can associate functorially an $n$-semisimplicial $\C$-object $X_{\bu}$ and an augmentation $\eps:X_{\bu}\rightarrow X_{\emptyset}$. We set:

\begin{equation*}
X_k:=\coprod_{|I|=k+1}X_I\hspace{0.5cm}\text{for }k=0,\ldots,n.
\end{equation*}
Let $\be:[s]\rightarrow [r]$ be a strictly increasing map (in particular $r\geq s$). If $I\in\sq_{n+1}$ has cardinality $r+1$ we can write it as $I=\{i_0,\ldots,i_r\}$ with $i_0<\dots<i_r$; the set $J:=\be(I):=\{i_{\be(0)},\ldots,i_{\be(s)}\}$ is contained in $I$, so we have a morphism $d_{JI}:X_I\rightarrow X_J$. We can now define the morphism

\begin{equation*}
d_{\be}:X_r\rightarrow X_s\hspace{0.5cm}\text{s.t. }(d_{\be})_{|X_I}=d_{\be(I)I}.
\end{equation*}
Since for any $I\in\sq_{n+1}$ we have a morphism $d_{\emptyset I}:X_I\rightarrow X_{\emptyset}$ we obtain the desired augmentation by setting $\eps_{|X_I}:=d_{\emptyset I}$.
\end{enumerate}
\end{rmk}

\begin{defn}
The category $TopAbSh$ has objects the pairs $(X,\F)$ where $X$ is a topological space and $\F$ is a sheaf of abelian groups on $X$, and as morphisms the pairs $(f,f^{\#}):(X,\F)\rightarrow (Y,\G)$ where $f:X\rightarrow Y$ is a continuous function and $f^{\#}:\G\rightarrow f_*\F$ is a morphism of sheaves of abelian groups on $Y$. A \textit{sheaf of abelian groups} $\F^{\bu}$ (resp. $\F^{\sq}$) on an $n$-semisimplicial space (resp. on an $n$-cubical space) is just an $n$-semisimplicial (resp. $n$-cubical) $TopAbSh$-object. In a similar manner, we can define complexes and resolutions of sheaves of abelian groups on $n$-semisimplicial or $n$-cubical spaces.
\end{defn}

Given a sheaf of abelian groups $\F^{\bu}$ on an $n$-semisimplicial space $X_{\bu}$, it is possible to define the cohomology of $X_{\bu}$ with values in $\F^{\bu}$; indeed, using the Godement resolutions of each $\F^m$ and differentials coming from the face maps of $\tr_n$, one obtains a double complex $F^{\bu,\bu}$ and sets

\begin{equation}
H^k(X_{\bu},\F^{\bu}):=H^k(s(F^{\bu,\bu}))
\end{equation}
where $s(F^{\bu,\bu})$ is the simple complex associated to $F^{\bu,\bu}$.

\begin{rmk}
If $Y$ is a constant $n$-semisimplicial space, any sheaf of abelian groups on $Y$ will be denoted by $\F$ and not by $\F^{\bu}$; likewise, the cohomology groups of $Y$ with values in $\F$ will be denoted by $H^k(Y,\F)$.
\end{rmk}

Assume that $\eps:X_{\bu}\rightarrow Y$ is an augmented $n$-semisimplicial space and $\F^{\bu}$ is a sheaf of abelian groups on $X_{\bu}$. The sheaves $\eps_*\C_{Gdm}^p(\F^q)$ form a double complex of sheaves of abelian groups on $Y$, whose associated simple complex gives

\begin{equation}
R\eps_*\F^{\bu}:=s[\eps_*\C_{Gdm}^{\bu}(\F^{\bu})].
\end{equation}
One can prove that the hypercohomology of the latter complex coincides with the cohomology of $X_{\bu}$ with values in $\F^{\bu}$, i.e.

\begin{equation}\label{eq:isocohom}
\mathbb{H}^k(Y,R\eps_*\F^{\bu})=H^k(X_{\bu},\F^{\bu})\hspace{0.5cm}\text{for any }k.
\end{equation}

\begin{defn}
\cite[Definition 5.3.2]{De} An augmented $n$-semisimplicial space $\eps:X_{\bu}\rightarrow Y$ is \textit{of cohomological descent} if for any sheaf of abelian groups $\F$ on $Y$ the natural adjunction morphism

\begin{equation}\label{eq:CohomDesc}
\F\rightarrow R\eps_*\eps^{-1}\F
\end{equation}
is an isomorphism in $D_+(Sh(Y))$. 
\end{defn}

\begin{rmk}
\begin{enumerate}
\vspace{1mm}
\item Observe that if $\eps:X_{\bu}\rightarrow Y$ is of cohomological descent then $H^k(Y,\F)\simeq H^k(X_{\bu},\eps^{-1}\F)$ for any $k$ and for any sheaf of abelian groups $\F$ on $Y$.
\item If $X_{\sq}$ is an $(n+1)$-cubical space and $\F^{\sq}$ is a sheaf of abelian groups on $X$, by point 3. of Remark \ref{rmk:cubobj} to this data we can associate an augmented $n$-semisimplicial space $\eps:X_{\bu}\rightarrow X_{\emptyset}$ and a sheaf of abelian groups $\F^{\bu}$ on it. We set

\begin{equation}\label{eq:cono}
C^{\bu}(X_{\sq},\F^{\sq}):=\text{Cone}^{\bu}[\F^{\emptyset}\rightarrow R\eps_*\eps^{-1}\F^{\emptyset}].
\end{equation}
\end{enumerate}
\end{rmk}

From now on, we will take for $\C$ the category whose objects are reduced separated schemes of finite type over $\CC$, which we will simply call varieties, and whose morphisms are morphisms of schemes; this is not fully consistent with the existing literature, in which the term `algebraic variety' is usually reserved for \textit{integral} separated schemes of finite type over some field. 

\begin{defn}
\begin{enumerate}
\item An augmented $n$-semisimplicial variety is of cohomological descent if this is the case for the associated augmented $n$-semisimplicial space.
\item Let $X$ be a variety. An $n$-\textit{semisimplicial resolution} of $X$ is an $n$-semisimplicial variety $\eps:X_{\bu}\rightarrow X$ augmented to $X$ such that all maps $X_m\rightarrow X$ are proper, all $X_m$ are smooth and $\eps$ is of cohomological descent.
\item An $(n+1)$-cubical variety is of cohomological descent (resp. a \textit{cubical hyperresolution}) if the associated augmented $n$-semisimplicial variety is of cohomological descent (resp. an $n$-semisimplicial resolution). By the definition of cohomological descent and (\ref{eq:cono}), $X_{\sq}$ is of cohomological descent if and only if $C^{\bu}(X_{\sq},\F^{\sq})$ is acyclic for any sheaf of abelian groups $\F^{\sq}$ on $X_{\sq}$.
\end{enumerate}
\end{defn}

We want to show that any variety $X$ admits an $m$-cubical hyperresolution $X_{\sq}$ for some $m$.

\begin{defn}
\begin{enumerate}
\item A \textit{proper modification} of a variety $S$ is a proper morphism $f:X\rightarrow S$ such that there exists $U\subset S$ open and dense for which $f$ induces an isomorphism $f^{-1}(U)\rightarrow U$; a \textit{resolution} of $S$ is a proper modification with $X$ smooth.
\item The \textit{discriminant} of a proper morphism of varieties $f:X\rightarrow S$ is the minimal closed subset $D\subset S$ such that $f$ induces an isomorphism $X\setminus f^{-1}(D)\rightarrow S\setminus D$.
\end{enumerate}
\end{defn}  

The notions of proper modification, resolution and discriminant extend immediately to $n$-cubical varieties and morphisms thereof. By \cite[Th\'eor\`eme I.2.6]{GNPP} any $n$-cubical variety $X_{\sq}$ admits a resolution, which is constructed by separating and resolving the irreducible components of each $X_I$ and then `patching together' the pieces in a way prescribed by the cubical structure of $X_{\sq}$.

\begin{defn}
Let $f:X_{\sq}\rightarrow S_{\sq}$ be a proper modification (resp. resolution) of an $n$-cubical variety. A \textit{discriminant square} (resp. \textit{resolution square}) for $f$ is a commutative diagram

\[\xymatrixcolsep{1.5cm}\xymatrixrowsep{1.5cm}\xymatrix{
E_{\sq} \ar[r]^j \ar[d] & X_{\sq} \ar[d]^f\\
D_{\sq} \ar[r]^i                  & S_{\sq}}
\]
where the horizontal maps are closed immersions and $f$ induces an isomorphism between $X_{\sq}-j(E_{\sq})$ and $S_{\sq}-i(D_{\sq})$ (i.e. $i(D_{\sq})$ contains the discriminant of $f$).
\end{defn}

\begin{lem}
\cite[Lemma 5.20]{PeSt} The $(n+2)$-cubical variety defined by a discriminant square for a proper modification of an $n$-cubical variety is of cohomological descent.
\end{lem}

We can now state the main result we shall need on cubical hyperresolutions:

\begin{thm}\label{thm:iperrescub}
Any variety $X$ admits an $(n+1)$-cubical hyperresolution $X_{\sq}$.
\end{thm}

\begin{proof}
A full proof can be found in \cite[Th\'eor\`eme I.2.15]{GNPP} or in \cite[Thereom 5.26]{PeSt}; here we are only interested in sketching how such a cubical hyperresolution can be constructed.

\begin{itemize}
\item Take a resolution $\pi:\tilde{X}\rightarrow X$ of $X$ and consider the $2$-cubical variety $X^{(1)}_{\sq}$ given by the associated resolution square:

\begin{equation*}
X^{(1)}_{\emptyset}:=X,\hspace{0.5cm}X^{(1)}_{\{0\}}:=\tilde{X},\hspace{0.5cm}X^{(1)}_{\{1\}}:=D,\hspace{0.5cm}X^{(1)}_{\{0,1\}}:=\pi_1^{-1}(D).
\end{equation*}
$X^{(1)}_{\sq}$ can be seen as a morphism of $1$-cubical varieties $f^{(1)}:Y_{\sq}^{(1)}\rightarrow Z_{\sq}^{(1)}$, with $Z_I$ smooth for $I\neq\emptyset$. 
\item Consider a resolution $\pi_2:\tilde{Y}_{\sq}^{(1)}\rightarrow Y_{\sq}^{(1)}$ and the corresponding resolution square; we obtain the diagram

\[\xymatrixcolsep{1.5cm}\xymatrixrowsep{1.5cm}\xymatrix{
E_{\square}^{(1)} \ar[r] \ar[d] & \tilde{Y}_{\square}^{(1)} \ar[d]^{\pi_2} \ar[rd]^{f^{(1)}\circ\pi_2} &\\
D_{\square}^{(1)} \ar[r]        & Y_{\square}^{(1)} \ar[r]^{f^{(1)}} & Z_{\square}^{(1)}.}
\]
The outer commutative square of $1$-cubical varieties can be considered as a $3$-cubical variety $X_{\square}^{(2)}$ i.e. as a morphism of $2$-cubical varieties $f^{(2)}:Y_{\sq}^{(2)}\rightarrow Z_{\sq}^{(2)}$, with $Z_I$ smooth for $I\neq\emptyset$.
\item Repeat the previous step enough times.
\end{itemize}
\end{proof}

Observe that if we take for $\C$ the category of $n$-cubical varieties and consider $X_{\sq}\in\C$, we can still apply the construction of Theorem \ref{thm:iperrescub} to $X_{\sq}$: at each step we obtain an $m$-cubical variety whose entries are $n$-cubical varieties. More precisely, Theorem \ref{thm:iperrescub} implies the following:

\begin{thm}\label{thm:iperrescubcub}
Any $n$-cubical variety $X_{\sq}$ admits a hyperresolution by an $m$-cubical variety $Y_{\sq}$ whose entries are $n$-cubical varieties.
\end{thm}

\begin{rmk}\label{rmk:iperrescubcub}
Assume that $X_{\sq}=\{X_{\emptyset},X_{\{0\}},X_{\{1\}},X_{\{01\}}\}$ is a $2$-cubical variety and $Y_{\sq}$ is an $m$-cubical hyperresolution of $X_{\sq}$, then $Y_{\sq}$ can be thought of as a $2$-cubical variety $Y'_{\sq}=\{Y'_{\emptyset},Y'_{\{0\}},Y'_{\{1\}},Y'_{\{01\}}\}$ of $(m-2)$-cubical varieties; by construction, for any $I\in\sq_2$ we have that $Y'_I$ is an $(m-2)$-cubical hyperresolution of $X_I$.
\end{rmk}

In \cite{GNPP} cubical hyperresolutions were used to define a cohomology theory for possibly singular algebraic varieties; namely:

\begin{defn}\label{defn:DeRhamComplex}
Let $X$ be a separated scheme of finite type over a field $\KK$ of characteristic zero, and let $\eps:X_{\sq}\rightarrow X$ be an $(n+1)$-cubical hyperresolution of $X$ together with its natural augmentation; the \textit{de Rham complex} and $k$-th algebraic de Rham cohomology group of $X$ are defined as

\begin{equation}
DR^{\bu}_X:=R\eps_*\Omega^{\bu}_{X_{\bu}}\hspace{1cm}H_{DR}^k(X):=\mathbb{H}^k(X,DR^{\bu}_X).
\end{equation}
If $V\subset X$ is a closed subset, the $k$-th algebraic de Rham cohomology group of $X$ with supports in $V$ is defined as

\begin{equation}
H_{DR,V}^k(X):=\mathbb{H}^k(V,R\Gamma_V DR^{\bu}_X).
\end{equation}
In both cases, $\eps:X_{\bu}\rightarrow X$ is the augmented $n$-semisimplicial resolution of $X$ associated to $X_{\sq}$.
\end{defn}

This cohomology theory coincides with the one developed by Hartshorne in \cite{H1} in case of an embeddable scheme $X$ over $\CC$, since the cohomology groups are defined as the hypercohomology of isomorphic complexes (see \cite[Th\'eor\`eme III.1.3]{GNPP}).

\begin{rmk}
The definitions of hyperresolution and de Rham complex $DR^{\bullet}_X$ given in \cite{GNPP} are actually different from the ones we presented here; if we denote by $\C$ the category of separated schemes of finite type over a field $\KK$ of characteristic zero, we have:

\begin{defn}\label{defn:hyperrescubGNPP}
\cite[D\'efinition I.3.2]{GNPP} If $X_{\sq}$ is an $n$-cubical hyperresolution of $X\in\C$, with its natural augmentation $X_{\sq}\rightarrow X$, then $X^*_{\sq}$ has a natural augmentation $\eps:X^*_{\sq}\rightarrow X$ to $X$ too. The latter is an \textit{$n$-cubical hyperresolution} of $X$.
\end{defn}

\begin{defn}\label{defn:DeRhamComplexBis}
\cite[D\'efinition III.1.10, Proposition III.1.12]{GNPP} If $X\in\C$ and $X^*_{\sq}$ is an $(n+1)$-cubical hyperresolution of $X$ with its natural augmentation $\eps:X^*_{\sq}\rightarrow X$, the \textit{de Rham complex} of $X$ is

\begin{equation*}
DR^{\bu}_X:=R\eps_*\Omega^{\sq}_{X^*_{\sq}}.
\end{equation*}
\end{defn}

Although different, the two definitions of the de Rham complex are equivalent. Indeed, pick $X\in\C$, let $X_{\sq}$ be an $(n+1)$-cubical hyperresolution of $X$ with its natural augmentation $\eps:X_{\sq}\rightarrow X$, and let $\eps:X_{\bu}\rightarrow X$ be the augmented $n$-semisimplicial resolution associated to it. Let $X^*_{\sq}$ be the augmented $(n+1)$-cubical hyperresolution of $X$ as in Definition \ref{defn:hyperrescubGNPP}, and denote by $\eps:X^*_{\sq}\rightarrow X$ its augmentation. In order to show that Definitions \ref{defn:DeRhamComplex} and \ref{defn:DeRhamComplexBis} are equivalent, we need to prove that

\begin{equation*}
R\eps_*\Omega^{\bu}_{X_{\bu}}\simeq R\eps_*\Omega^{\sq}_{X^*_{\sq}}
\end{equation*}
But this is a consequence of the construction we presented in point 3. of Remark \ref{rmk:cubobj}: indeed, that construction does not involve the entry $X_{\emptyset}$ of an $(n+1)$-cubical $\C$-object, hence all the entries of $X^*_{\sq}$ can be found in $X_{\bu}$ too (`bundled together' by the coproducts); moreover, the augmentation from the entries of $X_{\bu}$ to $X$ are combinations of the augmentations from the entries of $X^*_{\sq}$ to $X$. 
\end{rmk}

\section{Proof of Theorem 2}

Let us write resolution squares for $X$ and $Y$:

\begin{equation*}
S(X):=
\xymatrixcolsep{1.5cm}\xymatrixrowsep{1.5cm}\xymatrix{
D_X \ar[r] \ar[d] & \tilde{X} \ar[d]\\
\Sigma_X \ar[r] & X}\hspace{1cm}
S(Y):=
\xymatrixcolsep{1.5cm}\xymatrixrowsep{1.5cm}\xymatrix{
D_Y \ar[r] \ar[d] & \tilde{Y} \ar[d]\\
\Sigma_Y \ar[r] & Y.}
\end{equation*}

Consider the $m$-cubical hyperresolutions $H(Y)_{\sq}$ of $S(Y)$ and $H(X)_{\sq}$ of $S(X)$ provided by hypothesis (i). $H(Y)_{\sq}$ can be rewritten as a $2$-cubical variety of $(m-2)$-cubical varieties which, by Remark \ref{rmk:iperrescubcub}, are actually $(m-2)$-cubical hyperresolutions of the corresponding entries of $S(Y)$: 

\begin{equation*}
H(Y)_{\sq}=
\xymatrixcolsep{1.5cm}\xymatrixrowsep{1.5cm}\xymatrix{
{D_Y}_{\sq} \ar[r]^f \ar[d]^a & \tilde{Y}_{\sq} \ar[d]^b\\
{\Sigma_Y}_{\sq} \ar[r]^g & Y_{\sq}.}
\end{equation*}
Being a hyperresolution, $H(Y)_{\sq}$ is in particular of cohomological descent: hence, if $\underline{\CC}_{H(Y)_{\sq}}$ denotes the constant sheaf on $H(Y)_{\sq}$ then $C^{\bu}(H(Y)_{\sq},\underline{\CC}_{H(Y)_{\sq}})$ is acyclic, and the same is true of $C^{\bu}(H(Y)_{\sq},\underline{\CC}_{H(Y)_{\sq}})[2]$; by \cite[Corollary 5.28]{PeSt} we deduce the existence of an isomorphism

\begin{align*}
C^{\bu}(Y_{\sq},\underline{\CC}_{Y_{\sq}})\xrightarrow{\simeq}\text{Cone}^{\bu}&[Rb_*C^{\bu}(\tilde{Y}_{\sq},\underline{\CC}_{\tilde{Y}_{\sq}})\oplus Rg_*C^{\bu}({\Sigma_Y}_{\sq},\underline{\CC}_{{\Sigma_Y}_{\sq}})\xrightarrow{(C(a^{\#}),C(b^{\#}))}\\
&\xrightarrow{(C(a^{\#}),C(b^{\#}))} R(g\circ a)_*C^{\bu}({D_Y}_{\sq},\underline{\CC}_{{D_Y}_{\sq}})][-1].
\end{align*}
If we shift by $-1$ the short exact sequence of the cone over the morphism $(C(a^{\#}),C(b^{\#}))$ we obtain

\begin{align*}
0&\rightarrow R(g\circ a)_*C^{\bu}({D_Y}_{\sq},\underline{\CC}_{{D_Y}_{\sq}})[-1]\rightarrow\text{Cone}^{\bu}[Rb_*C^{\bu}(\tilde{Y}_{\sq},\underline{\CC}_{\tilde{Y}_{\sq}})\oplus Rg_*C^{\bu}({\Sigma_Y}_{\sq},\underline{\CC}_{{\Sigma_Y}_{\sq}})\xrightarrow{(C(a^{\#}),C(b^{\#}))}\\
&\xrightarrow{(C(a^{\#}),C(b^{\#}))} R(g\circ a)_*C^{\bu}({D_Y}_{\sq},\underline{\CC}_{{D_Y}_{\sq}})][-1]\rightarrow Rb_*C^{\bu}(\tilde{Y}_{\sq},\underline{\CC}_{\tilde{Y}_{\sq}})\oplus Rg_*C^{\bu}({\Sigma_Y}_{\sq},\underline{\CC}_{{\Sigma_Y}_{\sq}})\rightarrow 0
\end{align*}
so using the isomorphism above we get the short exact sequence of objects in $D_+(Sh(Y))$

\begin{equation*}
0\rightarrow R(g\circ a)_*C^{\bu}({D_Y}_{\sq},\underline{\CC}_{{D_Y}_{\sq}})[-1]\rightarrow C^{\bu}(Y_{\sq},\underline{\CC}_{Y_{\sq}})\rightarrow Rb_*C^{\bu}(\tilde{Y}_{\sq},\underline{\CC}_{\tilde{Y}_{\sq}})\oplus Rg_*C^{\bu}({\Sigma_Y}_{\sq},\underline{\CC}_{{\Sigma_Y}_{\sq}})\rightarrow 0.
\end{equation*}
Now, since the $(m-2)$-cubical hyperresolution $\eps:Y_{\sq}\rightarrow Y$ is of cohomological descent $C^{\bu}(Y_{\sq},\underline{\CC}_{Y_{\sq}})$ is acyclic; hence, if we denote by $Y_{\bu}$ the $(m-3)$-semisimplicial space associated to $S_{\sq}$ we can write the following isomorphism in $D_+(Sh(Y))$:

\begin{equation*}
\underline{\CC}_{Y}\xrightarrow{\simeq} R\eps_*\underline{\CC}_{Y_{\bu}}.
\end{equation*}
Since all elements of the $(m-3)$-semisimplicial variety $Y_{\bu}$ are smooth, in $D_+(Sh(Y_{\bu}))$ we have an isomorphism $\underline{\CC}_{Y_{\bu}}\xrightarrow{\simeq}\Omega^{\bu}_{Y_{\bu}}$, so we can substitute $R\eps_*\underline{\CC}^{\bu}_{Y_{\bu}}$ with $R\eps_*\Omega^{\bu}_{Y_{\bu}}$; the same can of course be done with the other $(m-2)$-cubical hyperresolutions in $H(Y)_{\sq}$.

In this way we obtain a short exact sequence of objects in $D_+(Sh(Y))$

\begin{equation}\label{eq:SESforY}
0\rightarrow R(g\circ a)_*DR_{D_Y}^{\bu}[-1]\rightarrow DR_Y^{\bu}\rightarrow Rb_*DR_{\tilde{Y}}^{\bu}\oplus Rg_*DR_{{\Sigma_Y}}^{\bu}\rightarrow 0
\end{equation}
which yields the long exact sequence of algebraic de Rham cohomology groups

\begin{equation}\label{eq:LECSY}
\dots\rightarrow H_{DR}^{\bu}(\Sigma_Y)\oplus H_{DR}^{\bu}(\tilde{Y})\rightarrow H_{DR}^{\bu}(D_Y)\rightarrow H_{DR}^{\bu+1}(Y)\rightarrow\cdots.
\end{equation}
We want to apply a similar argument to $H(X)_{\sq}$. We rewrite it as 

\begin{equation*}
H(X)_{\sq}=
\xymatrixcolsep{1.5cm}\xymatrixrowsep{1.5cm}\xymatrix{
{D_X}_{\sq} \ar[r]^f \ar[d]^a & \tilde{X}_{\sq} \ar[d]^a\\
{\Sigma_X}_{\sq} \ar[r]^g & X_{\sq}}
\end{equation*}
where each entry is an $(m-2)$-cubical hyperresolution of the corresponding entry of $S(X)$. Hypothesis (i) implies in particular that each $(m-2)$-cubical hyperresolution in $H(Y)_{\sq}$ can be embedded into the corresponding $(m-2)$-cubical hyperresolution in $H(X)_{\sq}$ as a closed $m$-cubical subvariety; if we pass to semisimplicial objects, we deduce the existence of natural closed immersions 

\begin{equation*}
H(Y)_{\bu}\hookrightarrow H(X)_{\bu}\hspace{0.5cm}Y_{\bu}\hookrightarrow X_{\bu}\hspace{0.5cm}{\Sigma_Y}_{\bu}\hookrightarrow {\Sigma_X}_{\bu}\hspace{0.5cm}\tilde{Y}_{\bu}\hookrightarrow \tilde{X}_{\bu}\hspace{0.5cm}{D_Y}_{\bu}\hookrightarrow {D_X}_{\bu}
\end{equation*}
and of the corresponding restriction of sections functors, which we shall denote by $\Gamma$.

Now we apply the same argument as before to the complex of sheaves on $H(X)_{\sq}$ given by $R\Gamma_{H(Y)_{\sq}}\underline{\CC}_{H(X)_{\sq}}$.

\begin{rmk}
We have the following commutative diagram of functors:

\[\xymatrixcolsep{1.5cm}\xymatrixrowsep{1.5cm}\xymatrix{
Sh({D_X}_{\bu}) \ar[r]^{\Gamma_{{D_Y}_{\bu}}} \ar[d]^{\eps_*} & Sh({D_X}_{\bu}) \ar[d]^{\eps_*}\\
Sh(D_X) \ar[r]^{\Gamma_{D_Y}} & Sh(D_X).}
\]
From this we deduce the equality of the total derived functors $R(\eps_*\circ\Gamma_{{D_Y}_{\bu}})=R(\Gamma_{D_Y}\circ\eps_*)$. But pushforwards preserve injective objects, and the same holds for $\Gamma_{{D_Y}_{\bu}}$ because ${D_Y}_{\bu}$ is closed in ${D_X}_{\bu}$; since injective objects are adapted to any functor, we obtain isomorphisms

\begin{equation}\label{eq:comm1}
R\eps_*\circ R\Gamma_{{D_Y}_{\bu}}\simeq R(\eps_*\circ\Gamma_{{D_Y}_{\bu}})=R(\Gamma_{D_Y}\circ\eps_*)\simeq R\Gamma_{D_Y}\circ R\eps_*.
\end{equation}
This commutativity holds for all the restriction of sections functors previously listed.
\end{rmk}

The $(m-2)$-cubical hyperresolution $\eps:X_{\sq}\rightarrow X$ is of cohomological descent so in $D_+(Sh(X))$ we have an isomorphism $\underline{\CC}_{X}\xrightarrow{\simeq} R\eps_*\underline{\CC}_{X_{\bu}}$. Moreover all elements of the $(m-3)$-semisimplicial variety $X_{\bu}$ are smooth, so in $D_+(Sh(X_{\bu}))$ we also have an isomorphism $\underline{\CC}_{X_{\bu}}\xrightarrow{\simeq}\Omega^{\bu}_{X_{\bu}}$. If we combine these facts we obtain isomorphisms

\begin{equation*}
R\Gamma_Y\underline{\CC}_X\simeq R\Gamma_Y R\eps_*\underline{\CC}_{X_{\bu}}\simeq R\Gamma_Y R\eps_*\Omega_{X_{\bu}}^{\bu}=R\Gamma_Y DR^{\bu}_X
\end{equation*}
that have counterparts for all the $(m-2)$-cubical hyperresolution in $H(X)_{\square}$. Thus we obtain the short exact sequence of objects of $D_+(Sh(X))$

\begin{equation}\label{eq:SESforX}
0\rightarrow R(g\circ a)_*R\Gamma_{D_Y} DR_{D_X}^{\bu}[-1]\rightarrow\Gamma_Y DR_X^{\bu}\rightarrow Rb_*\Gamma_{\tilde{Y}}DR_{\tilde{X}}^{\bu}\oplus Rg_*\Gamma_{\Sigma_Y}DR_{\Sigma_X}^{\bu}\rightarrow 0
\end{equation}
which yields the long exact sequence of algebraic de Rham cohomology groups with supports

\begin{equation}\label{eq:LECSX}
\dots\rightarrow H_{DR,\Sigma_Y}^{\bu}(\Sigma_X)\oplus H_{DR,\tilde{Y}}^{\bu}(\tilde{X})\rightarrow H_{DR,D_Y}^{\bu}(D_X)\rightarrow H_{DR,Y}^{\bu+1}(X)\rightarrow\cdots.
\end{equation}
Consider now the $m_1$-cubical hyperresolutions ${D_Y}_{\sq}$ of $D_Y$ and ${D_X}_{\sq}$ of $D_X$ provided by hypothesis (ii); we want to use the following result:

\begin{lem}\label{lem:tracemaps}
\cite[Lemma II.3.1]{H1} Let $\KK$ be a field of characteristic zero. Let $f:X\rightarrow Y$ be either a smooth morphism or a closed immersion of smooth schemes of finite type over $\KK$. Let $Z$ be a closed subscheme of $X$ such that the induced map $f:Z\rightarrow Y$ is a closed immersion. Then the trace map gives an isomorphism of complexes in $D_+(Sh(Y))$

\begin{equation*}
\text{Tr}_f:f_*R\Gamma_Z\Omega_X^{\bu}[2n]\rightarrow R\Gamma_Z\Omega_Y^{\bu}
\end{equation*}
where $n=\dim(X)-\dim(Y)$ and $\Gamma_Z$ denotes the restriction of sections functor. 
\end{lem}

Fix any $I\in\sq_{m_1}^*$ and consider any irreducible component $V$ of ${D_X}_I$; two things can happen:

\begin{enumerate}[(I)]
\item To $V$ there corresponds a unique irreducible component $U$ of ${D_Y}_I$ that admits a closed immersion $i:U\hookrightarrow V$ of codimension $1$.
\item To $V$ there corresponds no irreducible component of ${D_Y}_I$.
\end{enumerate}

In the first case we can use Lemma \ref{lem:tracemaps} with $Y=V$ and $Z=X=U$ to deduce the isomorphism

\begin{equation*}
i_*\Omega_U\xrightarrow{\simeq} R\Gamma_{U}\Omega_V[2]\hspace{0.5cm}\text{in }D_+(Sh(V)).
\end{equation*}
In the second case we have a closed immersion $\emptyset\hookrightarrow V$ and using Lemma \ref{lem:tracemaps} we find an isomorphism between trivial complexes. If we repeat the same reasoning for all irreducible components of ${D_X}_I$ we obtain an isomorphism 

\begin{equation}
i_*\Omega_{{D_Y}_I}\xrightarrow{\simeq} R\Gamma_{{D_Y}_I}\Omega_{{D_X}_I}^{\bu}[2]\hspace{0.5cm}\text{in }D_+(Sh({D_X}_I)).
\end{equation}
where $i$ is the closed immersion ${D_Y}_I\hookrightarrow {D_X}_I$. If we do the same for all $I\in\sq^*_{m_1}$ and then switch to semisimplicial objects, we obtain the isomorphism

\begin{equation}\label{eq:isosheaves}
i_*\Omega_{{D_Y}_{\bu}}^{\bu}\xrightarrow{\simeq} R\Gamma_{{D_Y}_{\bu}}\Omega_{{D_X}_{\bu}}^{\bu}[2]\hspace{0.5cm}\text{in }D_+(Sh({D_X}_{\bu})).
\end{equation}
where $i$ is the closed immersion ${D_Y}_{\bu}\hookrightarrow{D_X}_{\bu}$.

\begin{rmk}
Denote by $j$ the closed immersion $D_Y\hookrightarrow D_X$ and by $\eps_Y$ and $\eps_X$ the natural augmentations ${D_Y}_{\bu}\rightarrow D_Y$ and ${D_X}_{\bu}\rightarrow D_X$; we have a commutative diagram

\[\xymatrixcolsep{1.5cm}\xymatrixrowsep{1.5cm}\xymatrix{
Sh({D_Y}_{\bu}) \ar[r]^{i_*} \ar[d]^{{\eps_Y}_*} & Sh({D_X}_{\bu}) \ar[d]^{{\eps_X}_*}\\
Sh(D_Y) \ar[r]^{j_*} & Sh(D_X).}
\]
From this we deduce the equality of the total derived functors $R(j_*\circ{\eps_Y}_*)=R({\eps_X}_*\circ i_*)$. Both $i_*$ and $j_*$ are exact, because they are pushforwards of closed immersions, so they coincide with their derived functors; moreover, all pushforwards preserve injective objects, which are adapted to any functor. We thus obtain an isomorphism

\begin{equation}\label{eq:comm2}
R{\eps_X}_*\circ i_*=R{\eps_X}_*\circ Ri_*\simeq R({\eps_X}_*\circ i_*)=R(j_*\circ {\eps_Y}_*)\simeq Rj_*\circ R{\eps_Y}_*=j_*\circ R{\eps_Y}_*
\end{equation}
as functors from $D_+(Sh({D_Y}_{\bu}))$ to $D_+(Sh(D_X))$. 
\end{rmk}

If we apply to the sides of (\ref{eq:isosheaves}) the corresponding $R\eps_*$ and use (\ref{eq:comm2}) on the left-hand side and (\ref{eq:comm1}) on the right-hand side, we obtain an isomorphism between de Rham complexes $j_*DR^{\bu}_{D_Y}\xrightarrow\simeq R\Gamma_{D_Y}DR^{\bu}_{D_X}[2]$ in $D_+(Sh(D_X))$ and so an isomorphism between de Rham cohomology groups. Since we can repeat the previous reasoning for the $m_2$-cubical hyperresolutions of $\Sigma_Y$ and $\Sigma_X$ provided by hypothesis (ii), we obtain isomorphisms

\begin{align}\label{eq:morfismibelli}
\begin{split}
&H_{DR}^{\bu}(D_Y)\xrightarrow{\simeq} H_{DR,D_Y}^{\bu+2}(D_X)\\
&H_{DR}^{\bu}(\Sigma_Y)\xrightarrow{\simeq} H_{DR,\Sigma_Y}^{\bu+2c}(\Sigma_X).
\end{split}
\end{align}
These isomorphisms will allow us to relate the long exact sequences (\ref{eq:LECSY}) and (\ref{eq:LECSX}).

\paragraph{$\underline{c=1}$} For $k\geq 1$ we have the following diagram

\begin{equation}\label{eq:diag1}
\begin{tikzpicture}
  \matrix (m) [matrix of math nodes,row sep=3em,column sep=4em,minimum width=2em]
  {
   H_{DR}^{k-1}(\tilde{Y})\oplus H_{DR}^{k-1}(\Sigma_Y) & H_{DR,\tilde{Y}}^{k+1}(\tilde{X})\oplus H_{DR,\Sigma_Y}^{k+1}(\Sigma_X)\\
	 H_{DR}^{k-1}(D_Y) & H_{DR,DY}^{k+1}(D_X)\\
	 H_{DR}^k(Y) & H_{DR,Y}^{k+2}(X)\\
	 H_{DR}^k(\tilde{Y})\oplus H_{DR}^k(\Sigma_Y) & H_{DR,\tilde{Y}}^{k+2}(\tilde{X})\oplus H_{DR,\Sigma_Y}^{k+2}(\Sigma_X)\\
   H_{DR}^k(D_Y) & H_{DR,D_Y}^{k+2}(D_X)\\};
  \path[-stealth]
    (m-1-1) edge node [left] {$\al$} (m-2-1)
		(m-1-1) edge node [above] {$\simeq$} (m-1-2)
		(m-2-1) edge node [left] {$\beta$} (m-3-1)
		(m-2-1) edge node [above] {$\simeq$} (m-2-2)
		(m-3-1) edge node [left] {$\delta$} (m-4-1)
		(m-4-1) edge node [left] {$\sigma$} (m-5-1)
		(m-4-1) edge node [above] {$\simeq$} (m-4-2)
		(m-5-1) edge node [above] {$\simeq$} (m-5-2)
		(m-1-2) edge node [right] {$\al'$} (m-2-2)
    (m-2-2) edge node [right] {$\beta'$} (m-3-2)
		(m-3-2) edge node [right] {$\delta'$} (m-4-2)
		(m-4-2) edge node [right] {$\sigma'$} (m-5-2);
\end{tikzpicture}
\end{equation}
The two squares are commutative. Indeed, the trace maps are functorial by construction (see \cite[Chapter VI, Section 4.2]{H2}) so the same holds for the isomorphisms of cohomology groups they yield, which are the horizontal maps of this diagram; as the vertical maps are obtained from the hyperresolutions of $S(Y)$ and $S(X)$ they are functorial too, and this gives the commutativity of the squares. From this we deduce that $Ker(\delta)\simeq Ker(\delta')$ so we find an isomorphism $\theta_k:H_{DR}^k(Y)\xrightarrow{\simeq} H_{DR,Y}^{k+2}(X)$; moreover, we can choose the $\theta_k$ in such a way that \textit{all} the squares of (\ref{eq:diag1}) commute. 

Now, $X\setminus Y$ is affine so $H_{DR}^j(X\setminus Y)=0$ for $j\geq\dim(X)+1$ by \cite[Corollaire III.3.11(i)]{GNPP}; writing down the long exact sequence of algebraic de Rham cohomology groups associated to the pair $(X,X\setminus Y)$, we find that the morphism $H_{DR,Y}^{k+2}(X)\rightarrow H_{DR}^{k+2}(X)$ is surjective for $k+2=\dim(X)+1$ and an isomorphism for $k+2>\dim(X)+1$. If we pre-compose these morphisms with the corresponding $\theta_k$ we obtain morphisms $H_{DR}^k(Y)\rightarrow H_{DR}^{k+2}(X)$ that are surjective for $k=\dim(X)-1=\dim(Y)$ and isomorphisms for $k>\dim(Y)$; using the comparison theorem \cite[Theorem IV.1.1]{H1} we can conclude that these morphisms exist for singular cohomology too.

\paragraph{$\underline{c=0}$} In this case in order to have a diagram like (\ref{eq:diag1}) we need $k>2\dim(\Sigma_X)+1$, but this is the only difference with the previous case.

\section{Alexander polynomial and line arrangements}

By the works of Milnor \cite{Mi} and L\^e \cite{Le} we know in particular that if $f\in\CC[x_0,\ldots,x_n]$ is a homogeneous polynomial then the map $f:\CC^{n+1}\setminus f^{-1}(0)\rightarrow\CC^*$ is a smooth locally trivial fibration; its generic fibre, usually denoted by $F$, is called \textit{Milnor fibre}. To $F$ we can associate the geometric monodromy operator $h:F\rightarrow F$ and the induced algebraic monodromy operators $T_i:H^i(F,\CC)\rightarrow H^i(F,\CC)$. 

\begin{defn}
Let $C=V(f)\subset\PP^2$ be a reduced curve. The \textit{Alexander polynomial} of $C$ is the characteristic polynomial of $T_1$, and is denoted by $\Delta_C$.
\end{defn}

If $f$ has degree $d$ then $h$ is given by $\underline{x}\mapsto\eta_d\cdot\underline{x}$, where $\zeta_d$ is a primitive $d$-th root of unity; hence both $h$ and $T$ have order $d$, so $T$ is diagonalisable with roots of unity of order $d$ as eigenvalues. Moreover, the Milnor fibre of $C$ is a $d$-fold cover of $U:=\PP^2\setminus C$ and the geometric monodromy $h$ is a generator of the group of deck transformations of $F$; this implies that

\begin{equation}
\Delta_C(t)=(t-1)^{r-1}\prod_{1<k|d}\Phi_k(t)=(t-1)^{r-1}q(t).
\end{equation}
where $r$ is the number of irreducible components of $C$. We call $q(t)$ the non-trivial part of $\Delta_C(t)$, and say that $\Delta_C(t)$ is \textit{non-trivial} if $q(t)\neq 1$.

The most general tool for computing $\Delta_C$ is a formula by Libgober (see \cite{L1}) that involves type and relative position of the singularities of $C$; one can use it to verify one of the striking features of the Alexander polynomial: that it is rather hard to find curves for which it is non-trivial. This has led researchers to look for classes of curves for which the non-triviality of $\Delta_C$ could be detected by easier means, without the need to directly compute the whole polynomial. Line arrangements, which we will denote by $\Ab$, are one of these classes. The reason for this choice is two-fold: on the one hand, they are curves with the simplest possible singularities; on the other hand, one may try and take advantage of the combinatorial nature of such objects, encoded in their intersection semilattices $L(\Ab)$. 

Indeed, over the course of the years many examples and results have shown that the non-triviality of $\Delta_{\Ab}$ might be detected simply by looking at $L(\Ab)$; in order to present them properly, we need to introduce the notion of multinet \cite{FY,PS}:

\begin{defn}\label{defn:multinets}
Let $\Ab$ be a line arrangement, $\N$ denote a partition of $\Ab$ into $k\geq 3$ subsets $\Ab_1,\ldots,\Ab_k$, $m$ be a  `multiplicity function' $m:\Ab\rightarrow\NN$ and $\X$ be a subset of the multiple points of $\Ab$; consider moreover the following conditions:

\begin{enumerate}[(i)]
\item There exists $d\in\NN$ such that $\sum_{l\in\Ab_i}m(l)=d$ for all $i=1,\ldots,k$.
\item For any $l\in\Ab_i$ and $l'\in\Ab_j$ with $i\neq j$ we have $l\cap l'\in\X$.
\item For all $p\in\X$ the integer $n_p:=\sum_{l\in\Ab_i,p\in l}m(l)$ does not depend on $i$.
\item For all $i=1,\ldots,k$ and any $l,l'\in\Ab_i$, there is a sequence $l=l_0,\ldots,l'=l_r$ such that $l_{j-1}\cap l_j\notin\X$.
\end{enumerate}

The couple $(\N,\X)$ is called:

\begin{itemize}
\item a \textit{weak }$(k,d)$-\textit{multinet} if it satisfies (i)-(iii).
\item a $(k,d)$-\textit{multinet} if it satisfies (i)-(iv).
\item a \textit{reduced }$(k,d)$-\textit{multinet} if it satisfies (i)-(iv) and $m(l)=1$ for all $l\in\Ab$.
\item a $(k,d)$-\textit{net} if it satisfies (i)-(iv) and $n_p=1$ for all $p\in\X$; if $d=1$, the $(k,1)$-net is called a \textit{trivial }$k$-net.
\end{itemize}

We call $\Ab_1,\ldots,\Ab_k$ the \textit{classes} of $\N$, $\X$ its \textit{base locus} and $d$ its \textit{weight}. If $(\N,\X)$ is a weak $(k,d)$-multinet on $\Ab$ and $p$ is a multiple point of $\Ab$, we define the \textit{support of $p$ with respect to $\N$} as

\begin{equation*}
supp_{\N}(p):=\{\al\in\{1,\ldots,k\}|p\in l\text{ for some }l\in\Ab_{\al}\}.
\end{equation*}
\end{defn}

Observe that the notion of multinet is a mathematically precise formalisation of the notion of symmetry.

\begin{center}
\begin{tikzpicture}[shorten >=1pt,auto,node distance=2cm,
       thick,main node/.style={circle,draw,solid,fill=black,inner sep=0pt,minimum width=4pt}]
\draw[red,dotted] (0.75,0.5) -- (3.25,5.5);
\draw[green] (0.5,1) -- (5.5,1);
\draw[blue,dashed] (2.75,5.5) -- (5.25,0.5);
\draw[green] (3,0.5) -- (3,5.5);
\draw[blue,dashed] (0.25,0.5) -- (4.5,3.33);
\draw[red,dotted] (1.5,3.33) -- (5.75,0.5);
\filldraw [black] (1,1) circle [radius=2pt];
\filldraw [black] (5,1) circle [radius=2pt];
\filldraw [black] (3,5) circle [radius=2pt];
\filldraw [black] (3,2.33) circle [radius=2pt];
\end{tikzpicture}

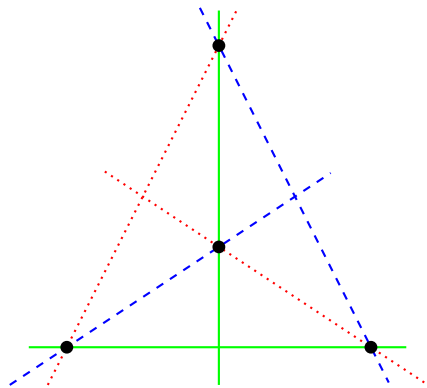
A $3$-net on the $A_3$ line arrangement.
\end{center}
We have the following result, which can be obtained as a consequence of \cite[Theorem 8.3]{PS} or combining \cite[Theorem 3.11]{FY} with \cite[Theorem 3.1(i)]{DP}:

\begin{thm}\label{thm:SuffCond}
If $\Ab$ admits a reduced multinet then its Alexander polynomial is non-trivial.
\end{thm}

This sufficient condition for the non-triviality of the Alexander polynomial of a line arrangement $\Ab$ is not necessary; however, all line arrangements with non-trivial Alexander polynomial known so far admit at least a $(k,d)$-multinet (and if the arrangement is non-central we have $k=3$ or $k=4$ only). This suggests that multinets somehow control the non-triviality of $\Delta_{\Ab}$; indeed, for some classes of line arrangements such a dependence has been established (\cite[Theorem 1.6, Theorem 1.2]{PS}):

\begin{thm}
Let $\Ab$ be a line arrangement with only double and triple points, then $\Delta_{\Ab}(t)=(t-1)^{|\A|-1}\Phi_3(t)^{\beta_3(\Ab)}(t)$ with $0\leq \beta_3(\Ab)\leq 2$ and $\beta_3(\Ab)\neq 0$ if and only if $\Ab$ admits a $3$-net.
\end{thm}

This result, together with the many examples gathered throughout the years, led Papadima and Suciu to formulate the following conjecture:

\begin{conj}
The Alexander polynomial of a line arrangement $\Ab$ has the form

\begin{equation*}
\Delta_{\Ab}(t)=(t-1)^{|\Ab|-1}\Phi_3(t)^{\beta_3(\Ab)}[\Phi_2(t)\Phi_4(t)]^{\beta_2(\Ab)}
\end{equation*}
\end{conj}

The numbers $\beta_i(\Ab)$ are the modular Aomoto-Betti numbers of $\Ab$ (see \cite[Section 3]{PS}), and they only depend on $L(\Ab)$ and $i$. Recent results \cite{MPP, D2, DSt} show that this conjecture is valid for all complex reflection arrangements.

\begin{rmk}
The only known arrangement with $\beta_2\neq 0$ is the Hesse arrangement: it can be constructed considering the nine inflection points of an elliptic curve and taking all lines that contain exactly three such points. We obtain an arrangement with twelve lines and nine point of order four with $\beta_2=2$.
\end{rmk}

The rest of this section is devoted to collecting some sparse results we shall need in the following one.

\begin{lem}\label{lem:ThomIso}
Suppose $f(x_0,\ldots,x_n)$ has an isolated singularity at the origin and $g(y_0,\ldots,y_n)$ has an arbitrary singularity at the origin. Call $F$, $G$ and $F\oplus G$ the Milnor fibres of $f$, $g$ and $f+g$ respectively, and denote by $T^i_f$, $T^i_g$ and $T^i_{f+g}$ the monodromy operators on the $i$-th cohomology groups. There is an isomorphism

\begin{equation*}
\tilde{H}^{n+k+1}(F\oplus G,\QQ)\simeq \tilde{H}^n(F,\QQ)\otimes\tilde{H}^k(G,\QQ)\hspace{0.8cm}\text{for }k=0,\ldots,n
\end{equation*}
respecting the monodromy operators: $T^{f+g}_{n+k+1}=T^f_n\otimes T^g_k$.
\end{lem} 

\begin{proof}
This is a consequence of \cite[Lemma 3.3.20, Corollary 3.3.21]{D1}.
\end{proof}

If $f(x_0,\ldots,x_n)=0$ is a homogeneous polynomial defining an isolated hypersurface singularity, the \textit{Steenbrink spectrum} of $f$ is the formal sum of rational numbers 

\begin{equation}
sp(f):=\sum_{\al\in\QQ}\al\nu(\al)
\end{equation}
where $\nu(\al)$ is the dimension of the $e^{-2\pi i\al}$-eigenspace of the monodromy operator acting on $Gr_F^{\floor{n-\al}}H^n(F)$. If $f$ has degree $d$ and weights $w_i$ then

\begin{equation}
\nu(\al)=\text{dim }M(f)_{(\al+1)d-w}
\end{equation}
where $M(f)$ is the Milnor algebra of $f$ and $w$ is the sum of the $w_i$'s. The spectrum is symmetric around $\frac{n-1}{2}$ and $\nu(\al)=0$ for $\al\notin (-1,n)$. 

\section{Proof of Theorem 3}

Consider a line arrangement $\Ab=V(f)$ of degree $n$ in $\PP^2$ having two multiple points $P_1$ and $P_2$ such that any line of the arrangement passes through $P_1$ or $P_2$. We call $p$ the multiplicity of $P_1$, $q$ the multiplicity of $P_2$ and assume, without loss of generality, that $p\geq q$; by our hypothesis on the arrangement, we have $p+q=n$ and all multiple points of $\Ab$ different from $P_1$ and $P_2$ have multiplicity two. Up to an isomorphism of $\PP^2$, we can assume that $P_1=(0:0:1)$ and $P_2=(0:1:0)$. 

\begin{rmk}
Assume $\Ab$ admits a weak $(k,d)$-multinet $(\N,\X)$ with classes $\Ab_1,\ldots,\Ab_k$. By definition of weak multinet and support, if $l\in\Ab$ then $|supp_{\N}(l)|\in\{1,k\}$. Thus, any double point of $\Ab$ is the intersection of lines belonging to the same class; but since all lines of $\Ab$ contain at least a double point, this means that all lines belong to the same class, which is a contradiction. Hence $\Ab$ does not admit weak multinets.
\end{rmk} 

A polynomial $f\in\CC[x_0,x_1,x_2]$ describing an arrangement of this type can be written as

\begin{equation*}
f=\prod_{i=1}^p(x_0-\lambda_i x_1)\prod_{i=1}^q(x_0-\mu_i x_2)
\end{equation*}
with at most one of the $\la_i,\mu_i$ equal to zero. Call $g=y^n+z^n$ and consider the threefold $X:=V(g-f)\subset\PP^4$: the natural projection map $\pi:X\rightarrow\PP^2$ s.t. $(y:z:x_0:x_1:x_2)\mapsto (x_0:x_1:x_2)$ is a rational map with discriminant $X\cap V(x_0,x_1,x_2)$.

Any hyperplane $H:=V(\al x_1-\beta x_0)\subset\PP^4$ cuts a surface from $X$; if we assume $\al\neq 0$ and call $s:=\beta/\al$ then this surface, which we denote by $Y_s$, is a hypersurface of $\PP^3$ defined by the polynomial 

\begin{equation*}
f_s:=y^n+z^n-h(s)x_0^p\prod_{i=1}^q(x_0-\mu_i x_2)\hspace{1cm}\text{ where }h(s):=\prod_{i=1}^p(1-\la_i s).
\end{equation*}
If $\al=0$ we denote the corresponding surface by $Y_{\infty}$, whose defining polynomial as hypersurface of $\PP^3$ is 

\begin{equation*}
f_{\infty}:=y^n+z^n-(-1)^n\left(\prod_{i=1}^p\la_i\prod_{i=1}^q\mu_i\right)x_1^px_2^q.
\end{equation*}
If we call $B$ the blow-up of $\PP^2$ at $P_1$ and set $X':=X\times_{\PP^2}B$ we obtain

\begin{equation*}
X'\simeq\{(y:z:x_0:x_1:x_2)\times(\al:\beta)\text{ s.t. }x_0\beta=x_1\al, y^n+z^n-f(x_0,x_1,x_2)=0\} 
\end{equation*}
and we can write the following diagram

\begin{center}
\begin{tikzpicture}
  \matrix (m) [matrix of math nodes,row sep=3em,column sep=4em,minimum width=2em]
  {
   X & \PP^2 \\
   X' & B\\
	 & \PP^1\\};
  \path[-stealth]
    (m-1-1) edge [dashed] node [above] {$\pi$} (m-1-2)
    (m-2-1) edge node {} (m-1-1)
		(m-2-1) edge node {} (m-2-2)
		(m-2-1) edge node [below] {$\psi$} (m-3-2)
		(m-2-2) edge node [right] {$\pi_2$} (m-1-2)
		(m-2-2) edge node [right] {$\pi_1$} (m-3-2);
\end{tikzpicture}
\end{center}
where $\pi_i$ is the projection from $B$ onto $\PP^i$, $\psi$ is given by $(y:z:x_0:x_1:x_2)\times (\al:\beta)\mapsto (\al:\beta)$ and the maps from $X'$ are the projections. The map $\psi:X'\rightarrow\PP^1$ is a fibration in surfaces: we have in fact

\begin{align*}
\psi^{-1}(1:s)\simeq Y_s\hspace{1cm}\psi^{-1}(0:1)\simeq Y_{\infty}.
\end{align*}
Both the threefold $X$ and the surfaces $Y_s$, $Y_{\infty}$ are singular, with singular loci given by

\begin{align*}
\Sigma_X&=\{P_p,P_q\}\cup\{(0:0:a:b:c)|(a:b:c)\text{ is a double point of }\Ab\}.\\
\Sigma_{Y_s}&=\begin{cases}
           \begin{array}{cl}
           \{P_p\}&\text{ if }h(s)\neq 0.\\
					 L_s&\text{ if }h(s)=0.
           \end{array}
					 \end{cases}\\
\Sigma_{Y_{\infty}}&=\{P_p,P_q\}.
\end{align*}
where $P_p:=(0:0:0:0:1)$, $P_q:=(0:0:0:1:0)$ and $L_s:=\{(0:0:a:as:b)\text{ s.t. }(a:b)\in\PP^1\}$.

Observe that the singularities of $X$ at the points in $\Sigma_X$ different from $P_p$ and $P_q$ are topologically equivalent to $y^n+z^n-v^2-w^2=0$; the singularity in $P_k$ is topologically equivalent to $y^n+z^n-v^k-w^k=0$ for $k=p,q$. The singularity of $Y_{\infty}$ (and $Y_s$ for $h(s)\neq 0$) in $P_p$ is topologically equivalent to $y^n+z^n-v^p=0$, while the one in $P_q$ is topologically equivalent to $y^n+z^n-v^q=0$.

Assume now $s_1$ and $s_2$ are not roots of $h(s)$, then we can find a diffeomorphism $Y_{s_1}\rightarrow Y_{s_2}$. Pick in fact $(y:z:x_0:s_1x_0:x_2)\in Y_{s_1}$, which satisfies $y^n+z^n-h(s_1)x_0^p\prod_{i=1}^q(x_0-\mu_i x_2)=0$: we can find $(\alpha y:\beta z:x_0:s_2x_0:x_2)\in Y_{s_2}$ for simple values of $\alpha$ and $\beta$. Namely, in order to have $(\alpha y:\beta z:x_0s_2x_0:x_2)\in S_{s_2}$ the equation $\alpha^n y^n+\beta^n z^n-h(s_2)x_0^p\prod_{i=1}^q(x_0-\mu_i x_2)=0$ must be satisfied; as $x_0^p\prod_{i=1}^q(x_0-\mu_i x_2)=\frac{y^n+z^n}{h(s_1)}$, we need to find $\alpha$ and $\beta$ satisfying

\begin{equation*}
\alpha^n y^n+\beta^n z^n-h(s_2)\frac{y^n+z^n}{h(s_1)}=0\Longleftrightarrow y^n\left(\alpha^n-\frac{h(s_2)}{h(s_1)}\right)=z^n\left(\beta^n-\frac{h(s_2)}{h(s_1)}\right)
\end{equation*}
and this gives $\alpha^n=\beta^n=\frac{h(s_2)}{h(s_1)}=:\gamma$.

If we call $\Delta:=\{(0:1)\}\cup\{(1:s)|h(s)=0\}$ we obtain a locally trivial fibration $T'-\psi^{-1}(\Delta)\rightarrow\PP^1-\Delta$, with the $Y_s$ with $h(s)\neq 0$ as generic fibre. We now compute the monodromy of $\psi$ around one of its special fibres, i.e. one of the $Y_s$ with $s\in\Delta$:

\begin{lem}
The geometric monodromy around a special fibre of $\psi$ is given by 

\begin{equation}
\phi:Y_s\rightarrow Y_s\text{ s.t. }(y:z:x_0:sx_0:x_2)\mapsto (\eta_ny:\eta_nz:x_0:sx_0:x_2)
\end{equation}
where $\eta_n$ is an $n$-th primitive root of unity.
\end{lem}

\begin{proof}
Assume the special fibre we are considering is $Y_{\frac{1}{\lambda_1}}$. Consider a loop $s(t)=\frac{1}{\lambda_1}+re^{2\pi it}$ around $\frac{1}{\lambda_1}$: by the above discussion, the diffeomorphism between $Y_{s(0)}$ and $Y_{s(t)}$ is governed by

\begin{equation*}
\gamma_t:=\frac{h(s(t))}{h(s(0))}=e^{2\pi it}\prod_{i=2}^p\frac{\lambda_1-re^{2\pi it}\lambda_1\lambda_i-\lambda_i}{\lambda_1-r\lambda_1\lambda_i-\lambda_i}
\end{equation*}
We can choose branch cuts for the $n$-th root function in such a way that, for $r$ small enough, the loop $s(t)$ remains in a zone of the complex plane in which the $n$-th root is a single-valued function. The only indeterminacy lies then in the term $e^{2\pi it}$; since we look for automorphisms $\phi_t:Y_{s(0)}\rightarrow Y_{s(t)}$ giving the identity for $t=0$, we deduce that the monodromy action $\phi$ on $Y_{s(0)}$ is given by $y\mapsto\eta_n y$, $z\mapsto\eta_n z$.
\end{proof}

Fix now an $s\neq\infty$ s.t. $h(s)\neq 0$ and write $Y:=Y_s$; what we want to do is the following: 

\begin{enumerate}[(1)]
\item Verify that the hypotheses of Theorem \ref{thm:risultato} are satisfied for $Y\subset X$, so that we can find a surjective morphism $H^2(Y)\twoheadrightarrow H^4(X)$.
\item Show that this morphism specialises to $H^2(Y)_{\text{prim}}^{T_{\phi}}\twoheadrightarrow H^4(X)_{\text{prim}}$, where $T_{\phi}$ denotes the algebraic monodromy.
\item Bound the dimension of $H^2(Y)_{\text{prim}}^{T_{\phi}}$, and thus of $H^4(X)_{\text{prim}}$, to deduce that the Alexander polynomial of $\Ab$ is trivial thanks to Lemma \ref{lem:ThomIso}.
\end{enumerate}

\subsection{Step 1}

Observe first that if $Y=X\cap H_s$ with $H_s=V(x_1-sx_0)$ then $\Sigma_Y=\Sigma_X\cap H_s$. We begin by finding explicit resolutions $\tilde{X}$ and $\tilde{Y}$ of $X$ and $Y$: since we have explicit equations for the singularities, this is just a matter of computations. It is straightforward to check that resolving $P_p$ yields as exceptional divisors:

\paragraph{\underline{if $p=q$}} On $\tilde{X}$ a smooth surface $E$ and $p$ disjoint planes $W_1,\ldots,W_p$, each of them intersecting $E$ in a line $L_i$, and on $\tilde{Y}$ a smooth curve $F=E\cap H$ such that $F\cap L_i=\emptyset$ for any $i=1,\ldots,p$.
  
\paragraph{\underline{if $p\neq q$}} On $\tilde{X}$
\begin{itemize}
\item A smooth surface $E$.
\item Planes $Z^{(t)}_i$ with $i=1,\ldots,p$ and $t=0,\ldots,r$ such that
			\begin{align*}
      \begin{array}{cl}
      Z^{(t_1)}_i\cap Z^{(t_2)}_j&=\begin{cases}
                                   \begin{array}{cl}
														       \text{a line}&\text{if }i=j\text{ and }t_1=t_2\pm 1.\\
														       \emptyset&\text{otherwise}.
														       \end{array}
														       \end{cases}\\
      Z^{(t)}_i\cap E&=\begin{cases}
                       \begin{array}{cl}
							         \text{a line }&\text{if }t=0.\\
							         \emptyset&\text{otherwise}.
								       \end{array}
							         \end{cases}
     \end{array}
     \end{align*}
\item Planes $Y^{(t)}_i$ with $i=1,\ldots,n$ and $t=0,\ldots,u$ such that
			\begin{align*}
      \begin{array}{cl}
      Y^{(t_1)}_i\cap Y^{(t_2)}_j&=\begin{cases}
                                   \begin{array}{cl}
														       \text{a line}&\text{if }i=j\text{ and }t_1=t_2\pm 1.\\
														       \emptyset&\text{otherwise}.
														       \end{array}
														       \end{cases}\\
      Y^{(t)}_i\cap E&=\begin{cases}
                       \begin{array}{cl}
							         \text{a line }&\text{if }t=0.\\
							         \emptyset&\text{otherwise}.
								       \end{array}
							         \end{cases}\\
			Y^{(t_1)}_i\cap Z^{(t_2)}_j&=\emptyset
     \end{array}
     \end{align*}
\end{itemize}
and on $\tilde{Y}$

\begin{itemize}
\item A smooth curve $F=E\cap H$.
\item Lines $K^{(t)}_i=Y^{(t)}_i\cap H$ with $i=1,\ldots,n$ and $t=0,\ldots,u$ such that
			\begin{align*}
      \begin{array}{cl}
      K^{(t_1)}_i\cap K^{(t_2)}_j&=\begin{cases}
                                   \begin{array}{cl}
														       \text{a point}&\text{if }i=j\text{ and }t_1=t_2\pm 1.\\
														       \emptyset&\text{otherwise}.
														       \end{array}
														       \end{cases}\\
      K^{(t)}_i\cap F&=\begin{cases}
                       \begin{array}{cl}
							         \text{a point}&\text{if }t=0.\\
							         \emptyset&\text{otherwise}.
								       \end{array}
							         \end{cases}
     \end{array}
     \end{align*}
\end{itemize}
The resolution of the points $P_q$ and $(0:0:a:b:c)$ s.t. $(a:b:c:)$ is a double point of $\Ab$, which belong to $X$ only, yield the same divisors as $P_p$ but with $p$ replaced by $q$ and $2$ respectively. 

We write the resolution squares of $X$ and $Y$ as 

\begin{equation*}
S(X)=
\xymatrixcolsep{1.5cm}\xymatrixrowsep{1.5cm}\xymatrix{
D_X \ar[r] \ar[d] & \tilde{X} \ar[d]\\
\Sigma_X \ar[r] & X}\hspace{1cm}
S(Y)=
\xymatrixcolsep{1.5cm}\xymatrixrowsep{1.5cm}\xymatrix{
D_Y \ar[r] \ar[d] & \tilde{Y} \ar[d]\\
\Sigma_Y \ar[r] & Y.}
\end{equation*}
We use Theorem \ref{thm:iperrescubcub} to obtain an $m$-cubical hyperresolution $H(X)_{\sq}$ of $S(X)$ for some $m$. Recall the process begins by taking a resolution of $S(X)$ as in \cite[Th\'eor\`eme I.2.6]{GNPP} (i.e. via separation and resolution of the irreducible components) and then considering the associated resolution square; this construction is iterated until we obtain an $m$-cubical hyperresolution. 

We observe the following: 

\begin{enumerate}[(a)]
\item irreducible components of $D_X$ and $D_Y$ and intersections thereof are smooth, and each irreducible component of $D_Y$ is an hyperplane section of an irreducible component of $D_X$.
\item $\Sigma_X$ and $\Sigma_Y$ are smooth, with the latter being a hyperplane section of the former. The same goes for $\tilde{X}$ and $\tilde{Y}$.
\end{enumerate} 

Given how $H(X)_{\sq}$ is constructed, these facts imply that considering in each entry of $H(X)_{\sq}$ the corresponding hyperplane section yields an $m$-cubical hyperresolution $H(Y)_{\sq}$ of $S(Y)$, so there is a natural closed immersion $H(Y)_{\sq}\hookrightarrow H(X)_{\sq}$; hence hypothesis (i) of Theorem \ref{thm:risultato} is satisfied.

Now we need to find suitable hyperresolutions of two of the four entries of $S(X)$ and $S(Y)$. For $\Sigma_X$ and $\Sigma_Y$ there is nothing to do, because they are smooth and hence coincide with their hyperresolutions: we thus immediately find a closed immersion $\Sigma_Y\hookrightarrow\Sigma_X$ of codimension zero. For $D_X$ and $D_Y$ we distinguish again two cases:

\paragraph{\underline{$p=q$}} To find a cubical hyperresolution of $D_X$ we use Theorem \ref{thm:iperrescub} (again, we separate and resolve its irreducible components $Z_i$); all $Z_i$ are smooth, and their pairwise intersections are either lines $L_j$ or the empty set. We find a resolution square (a $2$-cubical hyperresolution) like this

\begin{equation*}
{D_X}_{\sq}:=
\xymatrixcolsep{1.5cm}\xymatrixrowsep{1.5cm}\xymatrix{
K'_X \ar[r] \ar[d] & D'_X \ar[d]\\
K_X \ar[r] & D_X}
\end{equation*}
where

\begin{align*}
D'_X&=\coprod Z_i\\
K_X&=\bigcup L_j\\
K'_X&=(\coprod L^0_j)\coprod(\coprod L^1_j).
\end{align*}
If $L_j=Z_{i_0}\cap Z_{i_1}$ then $L^0_i$ denotes the line $L_i$ thought of as belonging to $Z_{i_0}$ and $L^1_i$ denotes the line $L_i$ thought of as belonging to $Z_{i_1}$.

Recall that since $p=q$ we have that $D_Y$ is a smooth curve $F$, so the procedure described in Theorem \ref{thm:iperrescub} returns $D_Y$ itself as cubical hyperresolution; we will thus need to construct a different cubical hyperresolution. First we consider the $1$-cubical variety given by the identity of $D_Y$, then we choose a point $Q\in D_Y$ and a point $Q'$ belonging to some $L_i=Z_{i_0}\cap Z_{i_1}$, and consider the 2-cubical variety

\begin{equation*}
{D_Y}_{\sq}:=
\xymatrixcolsep{1.5cm}\xymatrixrowsep{1.5cm}\xymatrix{
\{Q'\} \ar[r]^a \ar[d]^{id} & D_Y \ar[d]^{id}\\
\{Q'\} \ar[r]^c & D_Y}
\end{equation*}
where $a$ and $c$ send $Q'$ to $Q$. It is of cohomological descent because it is a discriminant square for $D_Y$; moreover, all of its entries are smooth and all the morphisms are proper. This means that ${D_Y}_{\sq}$ is a cubical hyperresolution of $D_Y$, and it is readily verified that it satisfies the hypothesis (ii) of Theorem \ref{thm:risultato}. 

\paragraph{\underline{$p\neq q$}} In this case it is enough to apply Theorem \ref{thm:iperrescub} to both $D_X$ and $D_Y$ to obtain the desired cubical hyperresolutions, thanks to point (a) above.

Hence hypothesis (ii) is satisfied too, and by part (b) of Theorem \ref{thm:risultato} we obtain a surjective morphism $H^2(Y)\twoheadrightarrow H^4(X)$.

\subsection{Step 2}

Denote by $\gamma$ the surjective morphism $H^2(Y)\twoheadrightarrow H^4(X)$ provided by Theorem \ref{thm:risultato}.

\begin{prop}\label{prop:monodromy}
We have $\gamma(H^2(Y))=\gamma(H^2(Y)^{T_{\phi}})$, so there is a surjective morphism

\begin{equation}\label{eq:gys1}
H^2(Y)^{T_{\phi}}\twoheadrightarrow H^4(X).
\end{equation}
\end{prop}

\begin{proof}
We denote by $\psi'$ the natural morphism $\tilde{X}\rightarrow\PP^1$ whose generic fibre is $\tilde{Y}$, by $\phi'$ the geometric monodromy on $\tilde{Y}$ and by $T_{\phi'}$ the induced automorphisms on the cohomology groups of $\tilde{Y}$. If we denote by $\tilde{\gamma}$ the usual Gysin morphism $H^2(\tilde{Y})\rightarrow H^4(\tilde{X})$ then by the global invariant cycle theorem (see for example \cite[Theorem 4.24]{V2}) we have

\begin{equation*}
\tilde{\gamma}(H^2(\tilde{Y}))=\tilde{\gamma}(H^2(\tilde{Y})^{T_{\phi'}}).
\end{equation*}
From the resolution square of $X$ we obtain the exact sequences of MHS (see \cite[Definition-Lemma 5.17]{PeSt})

\begin{align*}
\dots\rightarrow H^3(D_X)\rightarrow H^4(X)\rightarrow H^4(\tilde{X})\rightarrow\cdots;
\end{align*}
since the Hodge structure on $H^4(X)$ is pure by \cite[Proposition 6.33]{PeSt} and $H^3(D_X)$ has weights up to $3$, we deduce that $H^4(X)\rightarrow H^4(\tilde{X})$ is injective.

Now we observe that the diagram

\begin{equation}\label{eq:diagfin}
\xymatrixcolsep{1.5cm}\xymatrixrowsep{1.5cm}\xymatrix{
H^2(Y) \ar@{->>}[r]^{\gamma} \ar[d] & H^4(X) \ar@{^{(}->}[d]\\
H^2(\tilde{Y}) \ar@{->>}[r]^{\tilde{\gamma}} & H^4(\tilde{X})}
\end{equation}
is commutative. This can be read off the following diagram (there is a slight abuse of notation: we have switched to singular cohomology, but we maintain the names we gave to morphisms in the algebraic setting):

\[\xymatrixcolsep{1.5cm}\xymatrixrowsep{1.5cm}\xymatrix{
H^2(Y) \ar[r]^{\theta_2} \ar[d]^{\delta} & H_Y^4(X) \ar@{->>}[r] \ar[d]^{\delta'} & H^4(X) \ar@{^{(}->}[d]\\
H^2(\tilde{Y}) \ar[r]^{\simeq} & H_{\tilde{Y}}^4(\tilde{X}) \ar@{->>}[r] & H^4(\tilde{X}).
}
\]
The left square is commutative, because it is simply the equivalent, in singular cohomology, of the third square of diagram (\ref{eq:diag1}); the right square is commutative too, because the vertical maps are pullbacks and the horizontal maps come from the long exact sequences of the pairs $(X,X\setminus Y)$ and $(\tilde{X},\tilde{X}\setminus\tilde{Y})$ respectively, which are functorial. Since the compositions of the maps on the top and on the bottom give exactly the Gysin morphisms $\gamma$ and $\gamma'$, we obtain the commutativity of diagram \ref{eq:diagfin}.

The pullback morphism $H^2(Y)\rightarrow H^2(\tilde{Y})$ maps the subspace $V\subset H^2(Y)$ which is not $T_{\phi}$-invariant to the subspace $\tilde{V}\subset H^2(\tilde{Y})$ which is not $T_{\phi'}$-invariant, and the latter is sent to zero by $\tilde{\gamma}$ by the global invariant cycle theorem; since the diagram (\ref{eq:diagfin}) is commutative and $H^4(X)\rightarrow H^4(\tilde{X})$ is injective, we deduce that $\gamma(V)=0$.
\end{proof}

\begin{lem}\label{lem:primitive}
The morphism $H^2(Y)\twoheadrightarrow H^4(X)$ specialises to $H^2(Y)_{\text{prim}}\twoheadrightarrow H^4(Y)_{\text{prim}}$.
\end{lem}

\begin{proof}
Assume $H_0$ is the hyperplane of $\PP^4$ that cuts $Y$ from $X$, and choose another hyperplane $H$ of $\PP^4$ such that $H\cap \Sigma_X=\emptyset$; we can find a resolution of singularities $\pi_X:\tilde{X}\rightarrow X$ such that $\tilde{Y}:=\tilde{X}\cap\pi_X^{-1}(H_0)$ is smooth. If we call $\pi_Y:\tilde{Y}\rightarrow Y$ the restriction of $\pi_X$ to $\tilde{Y}$, we can write functorial morphisms

\begin{align*}
\begin{split}
\pi_Y^*:H^2(Y)\rightarrow H^2(\tilde{Y})\hspace{0.5cm} & \text{(pullback)}\\
\pi_X^*:H^4(X)\rightarrow H^4(\tilde{X})\hspace{0.5cm} & \text{(pullback)}\\
\tilde{\gamma}:H^2(\tilde{Y})\twoheadrightarrow H^4(\tilde{X})\hspace{0.5cm} & \text{(Gysin).}
\end{split}
\end{align*}
Since $\pi_X^{-1}(H)\simeq H$ and $\pi_Y^{-1}(H_0\cap H)\simeq H_0\cap H$ we deduce that $\tilde{\gamma}([\pi_Y^{-1}(H_0\cap H)])=[\pi_X^{-1}(H)]$; moreover, the functoriality of the pullback maps implies that $\pi_X^*([H])=[\pi_X^{-1}(H)]$ and $\pi_Y^*([H_0\cap H])=[\pi_Y^{-1}(H_0\cap H)]$.

The commutativity of (\ref{eq:diagfin}) now implies that $\gamma([H_0\cap H])$ can be written as $[H]+Ker(\pi_X^*)$, but since $\pi_X^*$ is injective it must be $\gamma([H_0\cap H])=[H]$; this proves our claim.
\end{proof}

The commutativity of (\ref{eq:diagfin}) actually allows us to further refine these results. Since $H^2(\tilde{Y})$ is a pure HS of weight $2$, the kernel of $H^2(Y)\rightarrow H^2(\tilde{Y})$ contains $W_1 H^2(Y)$; this, together with the injectivity of $H^4(X)\rightarrow H^4(\tilde{X})$ and the commutativity of (\ref{eq:diagfin}), implies that $W_1H^2(Y)\subset Ker(\gamma)$. The same holds true if we restrict first to primitive cohomology groups and then to the invariant part of $H^2(Y)_{\text{prim}}$ under the action of $T_{\phi}$, which proves that

\begin{equation}\label{eq:trick}
\gamma(H^2(Y)_{\text{prim}}^{T_{\phi}})=\gamma(W_2 H^2(Y)_{\text{prim}}^{T_{\phi}}).
\end{equation}

\subsection{Step 3}

If we call $U:=\PP^4\setminus X$ then from the long exact sequence of MHS associated to the pair $(\PP^4,X)$ we deduce

\begin{equation*}
\dots\rightarrow H^4(\PP^4)\rightarrow H^4(X)\rightarrow H^5_c(U)\rightarrow 0.
\end{equation*}
By using Poincar\'e duality and the isomorphism of homology and cohomology we obtain the isomorphism $H^5_c(U)\simeq H^3(U)^{\vee}$; since the map $H^4(\PP^4)\rightarrow H^4(X)$ is injective we obtain $H^4(X)_{\text{prim}}\simeq H^3(U)^{\vee}$. This implies in particular that $\dim H^4(X)_{\text{prim}}=\dim H^3(F_{g-f})^{T_{g-f}}$.

Similarly, if we call $U':=\PP^3\setminus Y$ from the long exact sequence of MHS associated to the pair $(\PP^3,Y)$ we deduce
\begin{equation*}
\dots\rightarrow H^2(\PP^3)\rightarrow H^2(Y)\rightarrow H^3_c(U')\rightarrow 0
\end{equation*}
and we obtain $H^2(Y)_{\text{prim}}\simeq H^3(U')^{\vee}$. Since we will need to study in detail the MHS on $H^2(Y)_{\text{prim}}$ we write the Poincar\'e duality isomorphism at the level of MHS: we have

\begin{equation}\label{eq:ISOMHS}
H^2(Y)_{\text{prim}}\simeq H^3(U')^{\vee}(-3).
\end{equation}
In order to simplify notations we call $V:=H^2(Y)_{\text{prim}}$. The isomorphism above implies the following equality of mixed Hodge numbers:

\begin{equation}\label{eq:isoMHN}
h^{p,q}(V)=h^{3-p,3-q}(H^3(U')).
\end{equation}
Since $V$ is a mixed Hodge substructure of $H^2(Y)$ it has weights $\leq 2$, and its Hodge filtration can be written as

\[\begin{array}{c}
0=F^3V\subset F^2V\subset F^1V\subset F^0V=V
\end{array}\]
while for $H^3(U')$ we have

\[\begin{array}{c}
0=F^4H^3(U')\subset F^3H^3(U')\subset F^1H^3(U')\subset F^0H^3(U')=H^3(U').
\end{array}\]
On $H^3(U')$ we also have the polar filtration:

\begin{equation*}
0=P^4 H^3(U')\subset\dots\subset P^1 H^3(U')=H^3(U').
\end{equation*}
Since the action of $T_{\phi}$ is compatible with all these filtrations, from (\ref{eq:isoMHN}), the inclusion $F^k H^3(U')\subseteq P^k H^3(U')$ given by \cite[Theorem 6.1.31]{D1} and the symmetry of mixed Hodge numbers we deduce

\begin{align}\label{eq:ineqs}
\begin{split}
h^{2,0}(V^{T_{\phi}})+h^{1,0}(V^{T_{\phi}})+h^{0,0}(V^{T_{\phi}})&\leq\dim P^3H^3(U')^{T_{\phi}}\\
h^{2,0}(V^{T_{\phi}})+2h^{1,0}(V^{T_{\phi}})+h^{0,0}(V^{T_{\phi}})+h^{1,1}(V^{T_{\phi}})&\leq\dim P^2H^3(U')^{T_{\phi}}.
\end{split}
\end{align}
We call now $R:=\CC[y,z,x_0,x_2]$, $f_Y\in R$ the polynomial defining $Y$ and $J_{f_Y}\subset R$ the associated Jacobian ideal; for $t=1,2,3$ we have maps 

\begin{equation}\label{eq:rec}
(R/J_{f_Y})_{tn-4}\twoheadrightarrow Gr_P^{4-t} H^3(U')=P^{4-t}H^3(U')/P^{5-t}H^3(U')
\end{equation}
Any class in $P^{k}H^3(U')$ has a representative of the form 

\begin{equation*}
\omega_h:=\frac{h\Omega}{f_Y^k}\hspace{1.55cm}\text{ with }h\in R_{kn-4}
\end{equation*}
(where $\Omega=ydz\wedge dx_0\wedge dx_2-zdy\wedge dx_0\wedge dx_2+x_0dy\wedge dz\wedge dx_2-x_2 dy\wedge dz\wedge dx_0$), and $T_{\phi}$ acts on it by multiplying $y$ and $z$ by $\eta_n$; this means that if $h(y,z,x_0,x_2)$ is an element of $(R/J_{f_S})_{kn-4}$ such that 

\begin{equation}\label{eq:condizione}
h(y,z,x_0,x_2)yzx_0x_2=h(\eta_n y,\eta_n z,x_0,x_2)\eta_n^2 yzx_0x_2
\end{equation} 
then the cohomology class $[\omega_h]\in H^3(U')$ is fixed by $T_{\phi}$. If we denote by $((R/J_{f_Y})_{tn-4})^{T_{\phi}}$ the elements of $(R/J_{f_Y})_{tn-4}$ satisfying condition (\ref{eq:condizione}), from (\ref{eq:rec}) we deduce 

\begin{equation}\label{eq:mappette}
(R/J_{f_Y})_{tn-4}^{T_{\phi}}\twoheadrightarrow Gr_P^{4-t} H^3(U')^{T_{\phi}}=P^{4-t}H^3(U')^{T_{\phi}}/P^{5-t}H^3(U')^{T_{\phi}}\hspace{0.5cm}\text{for }t=1,2,3.\\
\end{equation}
Let us compute the dimensions of the $(R/J_{f_Y})_{tn-4}^{T_{\phi}}$. A monomial $y^az^bx_0^cx_2^d$ satisfies condition (\ref{eq:condizione}) if and only if $a+b=kn-2$ for some $k\in\ZZ$ (*); since $J_{f_Y}$ contains $y^{n-1}$ and $z^{n-1}$, a monomial $y^az^bx_0^cx_2^d\in(R/J_{f_Y})_{tn-4}$ can satisfy (*) only for $k=1$; this implies in particular that $(R/J_{f_Y})_{n-4}^{T_{\phi}}=0$. From this we deduce that 

\begin{equation*}
Gr_P^3 H^3(U')^{T_{\phi}}=P^3 H^3(U')^{T_{\phi}}=0
\end{equation*}
which implies $Gr_P^2 H^3(U')^{T_{\phi}}=P^2 H^3(U')^{T_{\phi}}$; by (\ref{eq:ineqs}) we obtain

\begin{align*}
\dim V^{T_{\phi}}=h^{1,1}(V)^{T_{\phi}}\leq\dim Gr_P^2 H^3(U')^{T_{\phi}}.
\end{align*}
Since there are $n-1$ choices of non-negative $a,b<n-1$ that give $a+b=n-2$, we have $(n-1)^2$ monomials in $(R/J_{f_Y})_{2n-4}$ satisfying condition (\ref{eq:condizione}); this gives

\begin{equation}\label{eq:bound}
\dim V^{T_{\phi}}\leq (n-1)^2.
\end{equation}
Now we compute the dimension of $H^1(F_{y^n+z^n})$ by studying the Steenbrink spectra of the homogeneous isolated singularities of $\CC$ given by $y^n=0$ and $z^n=0$. 

We use the same notations as in Section 3. For $y^n$ we have $d=n$, $w=1$ and $M(y^n)=\CC\oplus\CC y\oplus\dots\oplus\CC y^{n-2}$, so the non-zero parts of $M(y^n)$ have weights $0,\dots,n-2$ and dimension $1$. In order to have $(\al+1)n-1=j$ for $j\in[0,n-2]$ we need $\al=\frac{j+1-n}{n}$, which implies $sp(y^n)=\sum_{j=0}^{n-2}(\frac{j+1-n}{n})$; this means the monodromy operator on $H^0(F_{y^n},\CC)$ has $n-1$ eigenspaces of dimension $1$ with eigenvalues $\zeta_n^a$ for $a\in[1,n-1]$ (and the same goes for $H(F_{z^n})$).

By Lemma \ref{lem:ThomIso} we deduce that $H^1(F_{y^n+z^n})$ has dimension $(n-1)^2$ and it is the direct sum of monodromy eigenspaces with eigenvalues $\eta_n^{a+b}$ for $a,b\in[1,n-1]$. The equality $\eta_n^{a+b}=\eta_n^k$ is satisfied by $n-2$ choices of the couple $(a,b)$ for $k\neq 0$, while for $k=0$ the choices are $n-1$: this means that in $H^1(F_{y^n+z^n})$ the fixed part under the monodromy action has dimension $n-1$, while all the other $n-1$ eigenspaces have dimension $n-2$.

Lemma \ref{lem:ThomIso} also allows us to write

\begin{equation*}
H^3(F_{g-f})^{T_{g-f}}=\bigoplus_{0\leq\al<1} H^1(F_g)_{1-\al}\otimes H^1(F_f)_{\al}
\end{equation*}
where the subscript $\al$ indicates the eigenspace relative to $e^{2\pi i\al}$. If we denote by $\epsilon_i$ the dimension of $H^1(F_f)_{\zeta_n^i}$ then $\epsilon_0=n-1$, so we can write the dimension of the right-hand side as

\begin{equation*}
(n-1)^2+\sum_{i=1}^{n-1}(n-2)\epsilon_i=(n-1)^2+(n-2)\sum_{i=1}^{n-1}\epsilon_i
\end{equation*}
From (\ref{eq:trick}) and (\ref{eq:bound}) we deduce that $H^3(F_{g-f})^{T_{g-f}}$ has dimension at most $(n-1)^2$, so $\epsilon_i=0$ for all $i\neq 0$, which means exactly that the Alexander polynomial of the arrangements we consider is trivial.

\begin{rmk}
Using \cite[Lemma 3.1]{PS} one can easily check that $\be_2(\Ab)=\be_3(\Ab)=0$ for $\Ab$ as in Theorem \ref{thm:AP}.
\end{rmk}

\begin{flushleft}
Department of Mathematics "Tullio Levi-Civita", University of Padua, Via Trieste 63, 35121 Padova, Italy.

\textit{E-mail address}: \verb+venturel@math.unipd.it+
\end{flushleft}

\end{document}